\documentclass{amsart}

\usepackage[english]{babel}
\usepackage{hyperref}
\usepackage{enumitem}
\usepackage{graphicx}
\usepackage{xcolor} % Colored comments
\usepackage{pgfplots}
\pgfplotsset{compat=1.18}

\newtheorem{theorem}{Theorem}[section]

\newtheorem{proposition}[theorem]{Proposition}
\newtheorem{corollary}[theorem]{Corollary}
\newtheorem{lemma}[theorem]{Lemma}
\newtheorem{example}[theorem]{Example}

\theoremstyle{remark}
\newtheorem{remark}[theorem]{Remark}

\numberwithin{equation}{section}

\begin{document}

%   TITLE
\title{Conic Pseudo-Finslerian Mechanical Systems}

%   FIRST AUTHOR

\author[B. O. Alves]{Benigno O. Alves}
\address{Benigno O. Alves \textup{(corresponding author)} \hfill\break\indent Instituto de Matem\'{a}tica e Estat\'{\i}stica, Universidade Federal da Bahia \hfill\break\indent  Rua Milton Santos, 40170-110 Salvador, Bahia, Brazil}
\email{benignoalves@ufba.br}

\thanks{The authors would like to thank Patricia Marçal and Miguel Angel Javaloyes for their valuable contributions to this work.}

%   SECOND AUTHOR 

%  SECOND AUTHOR
%  SECOND AUTHOR

\author[Glaene S. S. Mendonça]{Glaene S. S. Mendonça}
\address{Glaene S. S. Mendonça
\hfill\break\indent Faculdade de Ciências Exatas e Naturais -
\hfill\break\indent Departamento de Matemática e Estatística,
\hfill\break\indent Universidade do Estado do Rio Grande do Norte,
\hfill\break\indent Avenida Professor Antônio Campos, 2717, 59.625-620,
\hfill\break\indent Mossoró, Rio Grande do Norte, Brazil}
\email{glaenesantos@uern.br}

\thanks{This study was financed in part by the Coordena\c{c}\~ao de Aperfei\c{c}oamento de Pessoal de N\'{\i}vel Superior - Brasil (CAPES) - Finance Code 001.}

%   DATE
\date{\today}

%   KEYWORDS

%   ABSTRACT
\begin{abstract}
We develop a geometric framework for conservative mechanical systems modeled on conic pseudo-Finslerian manifolds, extending classical Riemannian and semi-Riemannian mechanics to anisotropic geometries. In this setting, we define a Jacobi-type pseudo-Finslerian metric and demonstrate that motions of fixed energy correspond, up to reparametrization, to its geodesics. We establish energy and momentum conservation under suitable symmetry assumptions via a generalized Noether’s theorem, and provide new analytical criteria for the completeness of both the Jacobi metric and the mechanical system. Furthermore, for mechanical systems with holonomic and non-holonomic constraints, we prove the existence and uniqueness of reaction forces satisfying d’Alembert’s principle. These results provide a unified geometric foundation for constrained and unconstrained anisotropic dynamics. 
\end{abstract}

%   SETUP
\maketitle

\section{Introduction}

A (pseudo-)Riemannian mechanical system is a triple $\mathcal{M} = (M,\mathcal{L},\mathcal{F})$, where $M$ is the configuration manifold, $\mathcal{L}(v_p) = \frac{1}{2}g_p(v,v)$ is the kinetic energy determined by a (pseudo-)Riemannian metric $g$, and $\mathcal{F} : TM \to TM$ is an external force. The system's trajectories satisfy Newton's equation $D_{\gamma'}\gamma' = \mathcal{F}(\gamma')$, where $D$ is the covariant derivative induced by $g$.

A fundamental case occurs when the system is conservative, i.e., $\mathcal{F}(v_p) = -\nabla \phi(p)$ for some potential function $\phi: M \to \mathbb{R}$. In such systems, trajectories are solutions of the Euler--Lagrange equations for $\widetilde{\mathcal{L}} = \mathcal{L} - \phi \circ \pi$, and the total mechanical energy $\mathcal{E} = \mathcal{L} + \phi \circ \pi$ is conserved, where $\pi : TM \to M$ is the canonical projection. Notably, the Maupertuis principle \cite{maupertuis1746}, refined by Jacobi, states that trajectories with fixed energy $e$ correspond, after reparametrization, to the geodesics of the Jacobi metric $g^J = (e - \phi)g$ defined on the region $M_e = \{p \in M \mid e - \phi(p) > 0\}$. This correspondence is a powerful tool in classical mechanics, providing a direct bridge between potential-driven dynamics and differential geometry \cite{godinho2012introduction,marsden2013introduction}.

The goal of this work is to extend this geometric formulation to mechanical systems in anisotropic media. In such media, physical properties --- and thus the fundamental tensor of kinetic energy --- depend not only on position but also on direction. This behavior cannot be modeled by a Riemannian metric, but it is naturally described within the framework of Finslerian geometry.

Finslerian geometry provides a powerful extension of Riemannian geometry, suitable for describing anisotropic phenomena across optics, thermodynamics, biomechanics, quantum field theory, and general relativity. Notable developments include the Zermelo navigation problem   \cite{bao2004zermelo,caponio2024wind,markvorsen2025time}, wind-type Finslerian structures \cite{caponio2024wind}, Lorentz--Finslerian spacetimes \cite{javaloyes2020definition}, applications to light propagation in birefringent media \cite{javaloyes2025generalized} and material anisotropy \cite{antonelli2005differential,clayton2023generalized,gibbons2011geometry,popov2022six}, robotics \cite{clayton2023generalized,ratliff2021generalized}, and frameworks for anisotropic wave propagation \cite{javaloyes2021applications} (e.g., the spread of wildfires or seismic wave propagation).

A conic pseudo-Finslerian mechanical system is a mechanical system whose kinetic energy $\mathcal{L} = \frac{1}{2}L$ is given by a conic pseudo-Finslerian metric $L$.
The existing literature on Finslerian mechanical systems, notably the works of Miron and collaborators \cite{bucataru1999nonlinear}, has primarily focused on the geometry of associated sprays, semisprays, and nonlinear connections.

Our central contribution lies in extending classical results for conservative mechanical systems to the conic pseudo-Finslerian setting. More precisely, a conic pseudo-Finslerian mechanical system is defined to be conservative if the external force is given by the anisotropic gradient $\mathcal{F}(v) = -\nabla^v\phi$, defined implicitly by $g_v(\nabla^v\phi, w) = d\phi(w)$ for all $w \in T_{\pi(v)}M$, where $g_v = \frac{1}{2}\mathrm{Hess}(L)(v)$ is the fundamental tensor of the metric $L$. We demonstrate that the trajectories of this system coincide with the Euler--Lagrange solutions for the modified Lagrangian $\widetilde{\mathcal{L}} = \frac{1}{2}L - \phi \circ \pi$ and, crucially, that the total mechanical energy $\mathcal{E} = \mathcal{L} + \phi \circ \pi$ is strictly conserved.

In \cite{maraner2019jacobi}, the author proved that low-energy solutions of a velocity-analytic Lagrangian coincide, up to reparametrization, with the geodesics of a Finslerian metric acting as a Jacobi metric. However, this result does not apply to the Lagrangian $\widetilde{\mathcal{L}} = \frac{1}{2}L - \phi \circ \pi$ when $L$ is a conic pseudo-Finslerian metric, since a general Finslerian metric fails to be of class $C^2$ on the zero section unless it reduces to a Riemannian metric.

The main result of this paper is the generalization of the Jacobi--Maupertuis principle to this setting. We introduce the Jacobi--Finslerian metric associated with the energy level $e$ as $L^J = (e - \phi)L$, defined on $M_e$. In Theorem~\ref{Jacobitheorem}, we prove that the geodesics of $L^J$ correspond, up to reparametrization, to the motions of the conservative conic pseudo-Finslerian mechanical system $\mathcal{M} = (M, \mathcal{L}, -\nabla^v\phi)$ with constant energy $e$.

This formulation provides a unified geometric foundation for anisotropic dynamics. Furthermore, we establish a momentum conservation theorem (Theorem~\ref{MomentumConservation}), study the completeness of the Jacobi metric and the mechanical system as well as dissipativity, and formulate the dynamics for holonomic and non-holonomic constraints, demonstrating the existence and uniqueness of reaction forces (Theorems~\ref{holonomiccontrains} and \ref{noholonomiccontrains}).

\section{Lagrangian Mechanics}\label{lagrangian}

In this section, we briefly recall the main concepts of Lagrangian mechanics. Let $M$ be a smooth manifold and $\pi: TM \to M$ the canonical projection. 
A curve $\gamma$  is \textbf{$\mathcal U$-admissible} if $\gamma'\in \mathcal U$. A \textbf{Lagrangian} on $M$ is a smooth function $\mathcal{L}: \mathcal{U} \subset TM \to \mathbb{R}$ defined on an open subset $\mathcal{U}$ of the tangent bundle. 
A  $\mathcal U$-admissible smooth curve $\gamma: I \to M$ is called an \textbf{$\mathcal{L}$-geodesic} (or \textbf{$\mathcal{L}$-motion}) if it is a critical point of the action functional
\[
A_{\mathcal{L}}(\gamma) = \int_I \mathcal{L}(\gamma'(t))\, dt.
\]

The \textbf{Legendre transformation} of $\mathcal{L}$ is the fiber derivative
$\ell = \ell^{\mathcal{L}} : \mathcal{U} \to T^*M,$ \[ 
\ell^{\mathcal{L}}(v)(w) = \left.\frac{d}{dt}\right|_{t=0} \mathcal{L}(v + t w),
\]
for all $v\in \mathcal{U}$ and $w\in T_{\pi(v)}M$.
The associated \textbf{energy function} (or \textbf{mechanical energy}) is
\[
\mathcal{E}(v) = \mathcal{E}^{\mathcal{L}}(v) = \ell^{\mathcal{L}}(v)(v) - \mathcal{L}(v),
\]
for each  $v \in \mathcal{U}$.

\begin{theorem}[Energy Conservation]\label{conservacaodaEnergia}
If $\mathcal L$ is a Lagrangian and $\gamma$ is an $\mathcal{L}$-geodesic, then $\mathcal{E}^{\mathcal{L}}(\gamma'(t))$ is constant along $\gamma$.
\end{theorem}

\begin{lemma}\label{L-phi} 
Let $\mathcal L$ be a Lagrangian and  $\phi : M \to \mathbb{R}$ be a smooth function. If  $\tilde{\mathcal{L}} = \mathcal{L} - \phi \circ \pi$, then  $\ell^{\tilde{\mathcal{L}}} = \ell^{\mathcal{L}}$ and $\mathcal{E}^{\tilde{\mathcal{L}}} = \mathcal{E}^{\mathcal{L}} + \phi \circ \pi$.
\end{lemma}

The following result is a version of Noether’s theorem for invariance under the flow of a vector field, not necessarily a Lie group action.

\begin{lemma}[Noether]\label{Noether}
Let $\mathcal{L}: \mathcal U \to \mathbb{R}$ be a Lagrangian invariant under the flow of a vector field $V \in \mathfrak{X}(M)$. 
Then for every $\mathcal{L}$-geodesic $\gamma$, $\ell^{\mathcal{L}}(\gamma')(V \circ \gamma)$ is constant along $\gamma$.
\end{lemma}

The \textbf{fundamental tensor} of a Lagrangian $\mathcal{L}: \mathcal{U} \to \mathbb{R}$ is the map assigning to each $v \in \mathcal{U}$ the symmetric bilinear form
\[
g_v^{\mathcal{L}}(u,w) 
= \operatorname{Hess}\mathcal{L}_v(u,w)
= \left.\frac{\partial^2}{\partial t\, \partial s}\right|_{t=s=0} 
\mathcal{L}(v + t u + s w),
\]
for each $u, w \in T_{\pi(v)}M$.

\begin{lemma}\label{proptfund1}
Let  $\mathcal{U} \subset TM \setminus 0$ be an open subset invariant under positive scalar multiplication and let$ \mathcal L$ be a Lagrangian, of class $C^2$ on $\mathcal U$, which is positively $2$-homogeneous. Then its fundamental tensor satisfies the following properties:
\begin{enumerate}\label{proptfundamenteal1}
    \item[(i)] $g_v$ is positively homogeneous of degree 0, that is, $g_{\lambda v}=g_{v}$ for any $\lambda >0;$
    \item[(ii)] $g_v(v,v)=2\mathcal L(v);$
    \item[(iii)] $g_v(v,w)=\ell_v(w)=d\mathcal L_v(w)$ for any $w\in T_{\pi(v)}M.$
\end{enumerate}
\end{lemma}
We say that $\mathcal{L}$ is \textbf{regular} if $g_v^{\mathcal{L}}$ is nondegenerate for all $v \in \mathcal{U}$.

\begin{theorem}\label{existencUnicit}
Let $\mathcal{L}$ be a regular Lagrangian. 
Then a smooth curve $\gamma = (x_1, \ldots, x_n)$ is an $\mathcal{L}$-geodesic if and only if it satisfies
\begin{equation}\label{equivaleeqEL}
\frac{d^2 x_i}{dt^2} 
+ \sum_j g^{ij}(v) 
\left(
  \frac{\partial^2 \mathcal{L}}{\partial v_j \partial x_i} \frac{d x_i}{dt} 
  - \frac{\partial \mathcal{L}}{\partial x_j}
\right) = 0,
\end{equation}
where $g^{ij}(v)$ denotes the inverse matrix of 
$g_{ij}(v) = g_v^{\mathcal{L}}\!\left(\frac{\partial}{\partial x_i}, \frac{\partial}{\partial x_j}\right)$.
In particular, for each $v \in \mathcal{U}$ there exists a unique maximal $\mathcal{L}$-geodesic $\gamma_v : I_v \to M$ such that $\gamma_v'(0) = v$ and $0 \in I_v$.
\end{theorem}

\section{Conic Pseudo-Finslerian Geometry}\label{Finsler}

In this section we summarize the concepts and results from the theory of conic pseudo-Finslerian manifolds that will be required in this work, mainly following \cite{javaloyes2014chern, javaloyes2021anisotropic}.

\subsection{ Conic pseudo-Finslerian metric}

Let $\mathcal{U} \subset TM \setminus 0$ be an open subset invariant under positive scalar multiplication. A smooth function $L: \mathcal{U} \to \mathbb{R}$ is called a \textbf{conic pseudo-Finslerian metric} if it satisfies:
\begin{enumerate}
    \item $L(\lambda v) = \lambda^2 L(v)$ for all $v \in \mathcal{U}$ and $\lambda > 0$;
    \item the \textbf{kinetic energy} $\mathcal{L} = \tfrac{1}{2} L$ is regular.
\end{enumerate} 
When $\mathcal{U} = TM \setminus \{0\}$, we simply say that $L$ is a \textbf{pseudo-Finslerian metric}. 

If moreover $g^{\mathcal L}_v$ is positive definite for all $v \in \mathcal{U}$, $L$ is a \textbf{conic Finslerian metric}. If, additionally, $\mathcal U=TM\setminus \{0\}$, then $L$ is a \textbf{Finslerian metric}. It is common to work with the positively homogeneous function of degree one $F$ satisfying $F^2 = L$, in which case $F$ is also referred to as a Finslerian metric.

\begin{lemma}\label{lema01}If $L$ is a conic pseudo-Finslerian metric on $M$, then 
$\ell^{\mathcal L}(v)(v)=L(v)$ and $\mathcal E^L(v)=\mathcal L(v)$ for each $v\in \mathcal U$. 
\end{lemma}

The \textbf{fundamental tensor} of a conic pseudo-Finslerian metric $L$ is the fundamental tensor of $\mathcal{L}=\frac{1}{2}L$, which will be denoted by $g^L$ or simply by $g$ when there is no risk of confusion. Moreover, we can verify that $g_v(v,v)=L(v)$ for all $v\in \mathcal U$.
Since $L$ is positively homogeneous of degree 2, it follows that
$g_{\lambda v}=g_v$
for any $\lambda>0$ and $v\in \mathcal U$.

We say that a subspace \( D \subset T_pM \) is \textbf{nondegenerate} if 
\(\mathcal{U} \cap D \neq \emptyset\) and, for every \(v \in \mathcal{U} \cap D\) 
the fundamental tensor \(g_v\) restricted to \(D_{\pi(v)}\) is nondegenerate.  
In this case, it follows that the restriction of \(g_v\) to the \textbf{$g_v$-orthogonal complement} 
$$D_v^{\perp}=\{ w\in T_pM ; g_v(w,u)=0 \textit{\ for all\ } u\in D\}$$ is also nondegenerate.  A submanifold \(N \subset M\) is called a \textbf{conic pseudo-Finslerian submanifold} of \((M,L)\) 
if \(T_pN\) is a nondegenerate subspace of \(T_pM\) for every \(p \in N\).  
Consequently, the restriction of \(L : \mathcal{U} \to \mathbb{R}\) to 
\(\mathcal{U} \cap TN\) defines a conic pseudo-Finslerian metric
\[
\widehat{L} = L|_{TN} : \mathcal{U} \cap TN \longrightarrow \mathbb{R}.
\]
It is straightforward to verify that the fundamental tensor of \(\widehat{L}\) 
coincides with the restriction of the fundamental tensor of \(L\) to \(TN\):
\[
g^{\hat L}_v(u,w) = g^L_v(u,w)
\qquad \] for all $
v \in \mathcal{U} \cap TN, $ and $u,w \in T_{\pi(v)}N.$

An \textbf{isometry} of $(M,L)$ is a diffeomorphism $f$ such that $L(df)=L$. A vector field $V$ is an \textbf{$L$-Killing vector field} if its  flow acts by local $L$-isometries of the pseudo-Finslerian metric.  

\begin{example}By \cite[Lemma 2.7]{alexandrino2019singular},
the isometry group of a Randers metric 
$R=\alpha+\beta$ coincides with the subgroup of $\alpha$-isometries preserving the 
$1$-form 
$\beta$.
For example, when 
$\alpha$ is the canonical Euclidean metric on 
$\mathbb R^3$
 and 
$\beta(x)=cx$, the only linear isometries are the rotations about the 
$x_3$-axis. This demonstrates the way anisotropy constrains the symmetry structure of the system. 
\end{example}

The fundamental tensor is an anisotropic $(0,2)$-tensor and its vertical derivative is the  following anisotropic $(0,3)$-tensor called \textbf{Cartan tensor}:
$$ C_v(u_1,u_2,u_3):=\frac{d}{dt}|_{t=0}g_{v+tu_1}(u_2,u_3)$$
for any $v\in \mathcal  U$ and $u_1,u_2,u_3\in T_{\pi(v)}M$.
For each $v\in \mathcal U$, $C_v$ is a symmetric trilinear tensor and satisfies
$C_v(v,.,.)=0$.

An \textbf{anisotropic vector field} on $\mathcal U$ is a smooth map $X:\mathcal U\to TM$ such that $ X(v)\in T_{v}T_{\pi(v)}M$ for all $v\in \mathcal U$. 
In local holonomic coordinates 
$$X= \sum_iX_i\frac{\partial}{\partial x_i}\circ \pi,$$
where $X_i:\mathcal U\to \mathbb R$ are smooth functions. 

The \textbf{anisotropic L-gradient} of a smooth function $f\in \mathcal F(M)$ is the anisotropic vector field $\nabla f:\mathcal U\to TM$, $v\mapsto \nabla ^vf$ implicitly given by 
$$g_v(\nabla^vf,w)=df_{\pi(v)}(w)$$
 for all $w\in T_{\pi(v)}M$.  In coordinates we have 
 $$\nabla^v f_p=\sum_{ik}g^{ik}(v)\frac{\partial f} {\partial x_i}(p)\frac{\partial } {\partial x_k}(p)$$
for each $p\in M$ and $v\in \mathcal U\cap T_pM$ where $g^{ij}(v)$ is the inverse matrix of $g_{ij}(v).$ 

Let $L$ be a conic pseudo-Finslerian metric and $\rho$ be a positive function on $M$. In particular, $\tilde{L} = (\rho \circ \pi) L$ defines a conic pseudo-Finslerian metric, called the \textbf{conformal} metric. If $\nabla f$ and $\widetilde{\nabla} f$ denote the anisotropic $L$-gradient and anisotropic $\tilde{L}$-gradient of a function $f$, respectively, then
\begin{eqnarray}\label{conformGrad}
    \nabla^v f = (\rho \circ \pi) \widetilde{\nabla}^v f.
\end{eqnarray}
for each $v\in \mathcal U$.
Indeed, it suffices to note that
\begin{eqnarray*}
    df(w) &=& g^{\tilde{L}}_{v}(\widetilde{\nabla}^v f, w) = \frac{1}{2} \frac{\partial^2}{\partial \partial s}\Big|_{t,s=0} \tilde{L}(v+t\widetilde{\nabla}^v f + sw)\\ &=& \frac{\rho \circ \pi}{2} \frac{\partial}{\partial s}\Big|_{s=0} L(v+t\widetilde{\nabla}^v f + sw) = (\rho \circ \pi) g^{L}_{v}(\widetilde{\nabla}^v f, w) = g^{L}_{v}((\rho \circ \pi)\widetilde{\nabla}^v f, w)
\end{eqnarray*}
for any $w\in T_{\pi(v)}M$, which implies equation \eqref{conformGrad}. 
\subsection{Chern Connection and Geodesic}

An \textbf{anisotropic linear connection} is a map 
\begin{eqnarray*}
    \nabla:\mathcal U\times \mathfrak X(M)\times \mathfrak X(M)&\to& TM \\
    (v,X,Y)&\mapsto & \nabla^v_XY
\end{eqnarray*}
such that, for any $X,Y\in \mathfrak X(M)$ the map $v\mapsto \nabla^v _X Y$ is an anisotropic vector field and $\nabla^V$ is a linear connection on $TU$ for each $V\in \mathfrak X(U)$  without singularities defined on an open subset $U \subset M$. We say that $\nabla$ is \textbf{torsion-free} if  
$$ \nabla^v_XY-\nabla^v_YX=[X,Y]$$
for every $X,Y\in \mathfrak X(M)$ and $v\in \mathcal U$. Given a local holonomic frame $\{\frac{\partial}{\partial x_i}\}$  on an open subset $U$, we define the Christoffel symbols of $\nabla$ as the functions $\Gamma_{ij}^k:\mathcal U\to \mathbb R$ satisfying $\nabla^v_{\frac{\partial}{\partial x_i}}\frac{\partial}{\partial x_j}=\sum_ k\Gamma^k_{ij}(v)\frac{\partial}{\partial x_k}$. Furthermore \begin{equation}
    \nabla^v_XY=\sum_k\Big\{X(y_k)_{\pi(v)}+\sum_{ij} \Gamma_{ij}^k(v)x_iy_j\Big\}\frac{\partial}{\partial x_k}(\pi(v))
\end{equation}
if $X=\sum_ix_i\frac{\partial}{\partial x_i}$ and $Y=\sum_iy_i\frac{\partial}{\partial x_i}$. 

The \textbf{Chern connection} on $L$ is the anisotropic linear connection which is torsion-free and \textbf{compatible with the metric} $L$ in the sense that 
$$0=(\nabla_X g)_V(Y,Z):=Xg_V(Y,Z)-g_V(\nabla^V_XY,Z))-g_V(Y,\nabla^V_XZ)-2C_V(Y,Z,\nabla^V_XV).$$

\begin{theorem}\label{ChernConnection}
There exists a unique Chern connection and it is characterized by  \textbf{Koszul-type formula}
 \begin{eqnarray*}		2g_V(\nabla^V_XY,Z)&=&Xg_V(Y,Z)-Zg_V(X,Y)+Yg_V(Z,X)\\
			&+&g_V([X,Y],Z)+ g_V([Z,X],Y)-g_V([Y,Z],X)\\
			&+& 2(-C_V(\nabla^V_X V,Y,Z) -C_V(\nabla^V_YV,Z,X) +C_V(\nabla^V_ZV,X,Y)),
		\end{eqnarray*}    
for each vector field without singularities $V$ and vector fields $X,Y,Z$.\end{theorem}

\begin{corollary}
    \label{lema2} Let $L$ be a conic pseudo-Finslerian metric and let $\tilde L = fL$ be a metric conformally related to $L$ where $f:M\to \mathbb R$ is a positive smooth function. Then 
    $$ \widetilde{\nabla}^V_VV= \nabla^V_VV+\frac{V(f)}{f}V-\frac{g_v(V,V)}{2f}\nabla^Vf.$$
    where $\nabla$ and $\widetilde{\nabla}$ are the Chern connections of $L$ and  $\tilde L$ respectively.  
\end{corollary}\begin{proof} The fundamental tensor of $\tilde L$ coincides with
$\tilde g_{v_p}=f(p)g_v,$
for all $v\in \mathcal U.$ Then by  Koszul-type formula and the properties of the Cartan tensor
    \begin{eqnarray*}			2fg_V(\widetilde{\nabla}^V_VV,Z)&=&fXg_V(V,V)-fZg_V(V,V)+fYg_V(Z,V)\\
		&+&fg_V([V,V],Z)+ fg_V([Z,V],V)-fg_V([V,Z],V)\\
			&+&V(f)g_V(V,Z)-Z(f)g_V(V,V)+V(f)g_V(Z,V)\\
            &=& 2f g_V(\nabla^V_VV,Z)
            +V(f)g_V(V,Z)-Z(f)g_V(V,V)+V(f)g_V(Z,V)\\
            &=& 
            2fg_V( \nabla^V_VV+\frac{V(f)}{f}V-\frac{g_v(V,V)}{2f}\nabla^Vf,Z).
		\end{eqnarray*}
This completes the proof.\end{proof}

A smooth curve $\gamma$ is  \textbf{$\mathcal U$-admissible} if $\gamma'\in \mathcal U$. 
We will denote by $\mathfrak X(\gamma)$  the space of smooth vector fields along $\gamma$.  An \textbf{anisotropic covariant derivative} in $\mathcal U$ along $\gamma$ is the map 
\begin{eqnarray*}
D_{\gamma}:\mathcal U\cap T_{\gamma}M\times \mathfrak X(\gamma)\to TM  
\end{eqnarray*}
such that $D_{\gamma}^{V}$ is covariant derivative along $\gamma$ for each $V\in \mathfrak X(\gamma)$ with $V(t)\in \mathcal{U}$ for all $t\in [a,b].$

\begin{lemma}Let $\nabla$ be the Chern anisotropic connection of a conic pseudo-Finslerian metric and $\gamma$ be an admissible smooth curve. 
     Then there is a unique anisotropic covariant derivative in $\mathcal U$ along a smooth curve $\gamma:[a,b]\to M$ with the
following property: if $X \in  \mathfrak X(M)$, then $D^v_{\gamma}(X\circ \gamma)=\nabla^v_{\gamma'(t)}X$ where $\pi(v)=\gamma(t)$.
\end{lemma}
   
We have that $D^{\lambda v}=D^{ v}$
for each $v\in \mathcal U$ and $\lambda>0$ and 
\begin{equation*}
\frac{d}{dt}g_{\gamma'}(X,Y)=g_{\gamma'}(D^{\gamma'}_{\gamma}X,Y)+g_{\gamma'}(X,D^{\gamma'}_{\gamma}Y)+2C_{\gamma'}(X,Y,D^{\gamma'}_{\gamma}\gamma'),
\end{equation*}
for each smooth curve $\gamma$ and vector fields $X,Y$ along $\gamma$.
In local holonomic coordinates 
$$D^v_{\gamma}X=\sum_k\Big\{x_k'(t)+\sum_{ij}\Gamma^k_{ij}(v)\gamma_i'(t)x_j(t)\Big\}\frac{\partial}{\partial x_k}(\gamma(t)),$$
if $X=\sum_kx_k\frac{\partial}{\partial x_k}\circ \gamma$ and $\pi(v)=\gamma(t)$.

The \textbf{geodesics of L} are the $\mathcal L$-geodesics, where $\mathcal L=\frac{1}{2}L$.  We can see that $\gamma$ is a geodesic of $L$ if and only if it satisfies the \textbf{geodesic equation}  $D^{\gamma'}_{\gamma}\gamma'=0$, or in coordinates  
\begin{equation}\label{geodesicequation}
\gamma_k''+\sum_{ij} \Gamma_{ij}^k(\gamma')\gamma_i'\gamma_j'=0 
\end{equation}
for all $k$.

\begin{lemma}\label{Lema1}
     An admissible curve $\gamma : I \subset \mathbb R \to M$ is a positive reparameterized geodesic of a conic pseudo-Finslerian metric $L$ if and only if it satisfies 
     \begin{equation*}
D^{\gamma'}_{\gamma}\gamma'=h(t)\gamma'(t),
     \end{equation*}
for some differentiable function $h:I\to \mathbb R$.\end{lemma}
\begin{proof}
Suppose that $\gamma(t)=\tilde\gamma\circ s(t)$. Then $\gamma'(t)=s'(t)\tilde \gamma'(s(t))$ and 
\begin{eqnarray}\label{Eqgeodcomp}
    D ^{ \gamma'}_{ \gamma} \gamma'&=&s'D_{\gamma}^{\gamma'}(\tilde \gamma'\circ s) +s''\tilde \gamma'\circ s =
    (s'(t))^2D^{s'\tilde \gamma'\circ s}_{ \tilde \gamma}(\tilde \gamma \circ s)+\frac{s''}{s'} \gamma'.
\end{eqnarray}
Thus, if $s'>0$ and $\tilde \gamma$ is a geodesic, then  
\begin{eqnarray*}
    D^{ \gamma'}_{ \gamma} \gamma'&=&(s')^2D^{\tilde \gamma'}_{ \tilde \gamma}\tilde \gamma' + \frac{s''}{s'} \gamma'
=\frac{s''}{s'} \gamma'.
\end{eqnarray*}
Conversely suppose that 
\begin{equation}\label{eqsupo} D_{\gamma}^{\gamma'}\gamma'=h\gamma' .
\end{equation} Consider the following function 
$$s(t):=\int \exp(\int h(t)dt)dt.$$
In particular, $s'=exp(\int h(t)dt)>0$. Let $\tilde \gamma$ be a smooth curve given by $\gamma=\tilde \gamma \circ s$. By construction 
$h=\frac{s''}{s'}$
and then by  Equations~\ref{Eqgeodcomp} and \ref{eqsupo} we can conclude that $\tilde \gamma$ is a geodesic. 
\end{proof}

\subsection{Anisotropic Second Fundamental Form}

A (regular) \textbf{distribution} of dimension $k$ on a manifold $M$ is a map 
$\mathcal D$ assigning to each point $p\in M$ a $k$-dimensional subspace 
$\mathcal D_p\subset T_pM$ such that for every 
$v\in \mathcal D$ there exist an open neighborhood $U$ of $\pi(v)$ and a vector field 
$V\in \mathfrak X(U)$ such that $V(\pi(v))=v$ and $V(p)\in \mathcal D_p$ for all 
$p\in U$. 
We denote by $\mathfrak X(\mathcal D)$ the set of smooth sections of $\mathcal D$, that is,
$V\in \mathfrak X(\mathcal D)$ if and only if $V\in \mathfrak X(M)$ and 
$V(p)\in \mathcal D_p$ for every $p\in M$. 

Let $\mathcal D$ be a distribution on a submanifold $N$ of $M$. An  \textbf{$\mathcal U$-anisotropic distribution} along $\mathcal D$ is a map $\mathcal S$ assigning to each $v\in \mathcal U\cap \mathcal D$ a subspace $\mathcal S_v\subset T_{\pi(v)}M$ such that for any anisotropic vector field $V\in \mathfrak X(\mathcal U)$, the corresponding family $\mathcal S_V$ defines a distribution.

If $L:\mathcal U\to\mathbb R$ is a conic pseudo-Finslerian metric and 
$\mathcal D$ is a smooth regular distribution on a submanifold $N$, 
we can define the \textbf{anisotropic orthogonal distribution} 
by
\[
\mathcal D^{\perp}_v := 
\{ w\in T_{\pi(v)}M : g_v(w,u)=0 \text{ for all } u\in \mathcal D_{\pi(v)}\}
\]
for each $v\in \mathcal U\cap \mathcal D$. If $\mathcal D$ is \textbf{nondegenerate}, for each $v\in \mathcal U\cap \mathcal D$ we have a direct sum decomposition $T_pM = \mathcal D_p \oplus \mathcal D^{\perp}_v,$
whose canonical projections will be denoted by 
$P_v:T_pM\to \mathcal D_p$ and 
$P_v^{\perp}:T_pM\to \mathcal D^{\perp}_v$. We then define the map $B:\mathcal U\cap \mathcal D\to  \mathcal D^{\perp}$ by
\begin{equation}\label{fakesegundaforma}
B(v)= P_{v}^{\perp}\big(\nabla _V^V V\big),
\end{equation}
where $V\in \mathfrak X(\mathcal U\cap \mathcal D)$ satisfies 
$V(\pi(v))=v$.

\begin{lemma}
The map $B$ is well defined.
\end{lemma}

\begin{proof}
Let $V,\tilde V\in \mathfrak X(\mathcal U\cap \mathcal D)$ with 
$V(\pi(v))=\tilde V(\pi(v))=v$, and let 
$W\in \mathfrak X(\mathcal D^{\perp}_V)$. 
Then, using the properties of the Chern connection, we obtain
\begin{eqnarray*}
g_v(\nabla_v^v(V-\tilde V),W)
&=& vg_V(V-\tilde V,W)
   - g_v(V(\pi(v))-\tilde V(\pi(v)),\nabla _v^vW)\\
& & +\,2C_v(V(\pi(v))-\tilde V(\pi(v)),W(\pi(v)),\nabla_{v}^vV) = 0.
\end{eqnarray*}
Since $g_v$ is nondegenerate on $\mathcal D^{\perp}_v$, it follows that 
$P_{v}^{\perp}(\nabla_V^VV)
= P_{v}^{\perp}(\nabla_{\tilde V}^{\tilde V}\tilde V)$, 
and hence $B$ is well defined.
\end{proof}

\begin{remark}
If $\mathcal D=TN$ and $L$ is a Riemannian metric, then the map $B$ coincides with the classical second fundamental form of the immersion $N\hookrightarrow M$.
\end{remark}

\subsection{Example: Zermelo Navigation}
The \textbf{Zermelo navigation problem} is the problem of finding the time-minimizing trajectory on a Riemannian manifold $(M,h)$ under the influence of a flow represented by a vector field $W$. These curves are the geodesics of the Zermelo Lagrangian $\mathcal{Z}=\tfrac{1}{2}Z$, where $Z$ is determined implicitly by 
\begin{eqnarray}\label{eqZermelo}
    \alpha_h\Big(\frac{v}{\sqrt{Z(v)}}-W_{\pi(p)}\Big)=1 ,
\qquad \forall,v\in TM,\end{eqnarray} 
where $\alpha^2_h(v)=h(v,v)$.
Define $\Lambda(p)=1-\alpha^2_h(W(p))$. 
If $\Lambda(p)>0$, then for every $v\in T_pM$ there exists a unique number $Z(v)>0$ such that 
$\frac{v}{Z(v)}\in \Sigma_p:=Z^{-1}(1)$, or equivalently, $Z(v)$ satisfies Equation~\ref{eqZermelo}.
On the open set $M_+ = \{\,p\in M \; ;\; \Lambda(p)>0\,\},$
Equation~\eqref{eqZermelo} defines a Finslerian metric $Z$, known as a \textbf{Randers metric with navigation data $(h,W)$}. 
Explicitly,
\begin{equation}\label{EqZermelo1}
    Z(v_p) = \frac{1}{\Lambda^2(p)} 
    \Big(
       \sqrt{\,\Lambda(p)\,h(v,v) + h(v,W_p)^2\,} 
       - h(v,W_p)
    \Big)^2.
\end{equation}

Now consider 
$M_- = \{\,p\in M \; ;\; \Lambda(p)<0\,\},$
and define the conic open subset
\[
\mathcal U := 
\{\, v\in TM_- \; ;\; 
      \Lambda(p)\,h(v,v) + h(v,W)^2 > 0 
      \text{ and } 
      h(v,W) > 0 
\}.
\]
For each $v\in \mathcal U$, there exist exactly two positive numbers $Z(v)$ and $Z_l(v)$ satisfying Equation~\eqref{eqZermelo}. 
The first one is again given by~\eqref{EqZermelo1}, while the second is
\[
Z_l(v_p) = 
\frac{1}{\Lambda^2(p)} 
\Big(
   -\sqrt{\,\Lambda(p)\,h(v,v) + h(v,W_p)^2\,} 
   - h(v,W_p)
\Big)^2.
\]
The metric $Z$ is a \textbf{conic Finslerian metric}, while $Z_l$ is a \textbf{conic pseudo-Finslerian metric}. 

If $\Lambda(p)=0$, then the cone 
$\mathcal U_p = \mathcal U\cap T_pM$ 
coincides with the half-plane containing $W(p)$, and for each $v\in \mathcal U_p$ there exists a unique $Z(v)>0$ satisfying~\eqref{eqZermelo}, namely,
\[
Z(v) = \frac{h(v,v)}{4\,h^2(v,W)}.
\]
This defines a \textbf{conic Finslerian metric}, which is a \textbf{Kropina metric}.

The conic Finslerian metric $Z$ is often referred to as the \textbf{Zermelo metric with navigation data} $(h,W)$, since its geodesics solve Zermelo’s navigation problem.

\medskip

More generally, an \textbf{$(\alpha,\beta)$-metric} is a Finslerian metric of the form
\[
Z = \alpha^2\,\phi^2\!\left(\frac{\beta}{\alpha}\right),
\]
where $\phi$ is a smooth positive function on an interval containing the range of $\beta/\alpha$, 
under suitable technical conditions ensuring the smoothness and strong convexity of $Z$.
This family includes Randers, Kropina, Matsumoto, and quadratic metrics. For more details,  see \cite{javaloyes2020definition}.

Let   $L:\mathcal U\to \mathbb R_+$ be a conic Finslerian metric and $W\in \mathfrak X(M)$. Consider the conic open set 
$$\mathcal V= \{ \lambda(v+W); v\in \mathcal U, L(v)=1\ \mbox{and} \ \lambda>0 \}.$$
A \textbf{Zermelo metric} with \textbf{navigation data} $(L,W)$ is a positive smooth function $ Z:\mathcal V\to \mathbb R$ given implicitly by the equation 
$$  L\Big(\frac{v}{\sqrt{Z(v)}}-W_{\pi(v)}\Big)=1$$
for each $v\in \mathcal V$.
This metric solves the Zermelo navigation problem with navigation data $(L,W)$.

\begin{lemma}\label{ZemeloConforme}
    Let $Z:\mathcal U\to \mathbb R$ be a conic Zermelo metric with navigation data $( L,W)$ and let $\Phi$ be
 a positive smooth function. Then $(\Phi\circ \pi )Z$ is the conic Zermelo metric with navigation data $((\Phi\circ \pi) L,\; \tfrac{1}{\sqrt{\Phi}}\,W).$
\end{lemma} 
\begin{proof}
 By definition
\[
1 = L\!\left(\frac{v}{\sqrt{Z(v)}} - W_p\right)
  = \Phi(p)\,  L\!\left(\frac{v}{\sqrt{\Phi(p)\,Z(v)}} - \frac{1}{\sqrt{\Phi(p)}}\,W_p\right)
\]    
for each $p\in M$ and  $v\in \mathcal U\cap T_pM$.
\end{proof}

  \begin{lemma}[\cite{alves2024isoparametric}]
\label{Randersimmersion}
Let $R$ be a Randers-Minkowski metric on $\mathbb R^n$ with navigation data $(h, W)$. If $V$ is a subspace, then the induced metric on $V$ is the Randers-Minkowski metric with navigation data $(\tilde h, W^{\top})$ given by
\begin{subequations}
\begin{equation*}
\tilde h := \lambda \, h\vert_{V \times V} \, ,
\end{equation*}
\begin{equation*}
W^{\top} := W - W^{\perp} \, ,
\end{equation*}
\end{subequations}
where $\lambda := \frac{1}{1 - h(W^{\perp}, W^{\perp})}$ and $W^{\perp}$ is the component of $W$ $h$-orthogonal to $V$.
\end{lemma}

\section{Conic Pseudo-Finslerian Mechanical Systems}\label{finslermechanical}

A \textbf{conic pseudo-Finslerian mechanical system} is a triple $\mathcal M=(M,\mathcal L,  \mathcal F)$ , where 
\begin{enumerate}
    \item $M$ is a differentiable manifold, called \textbf{configuration space};
    \item $\mathcal L=\frac{1}{2}L$ is the \textbf{kinetic energy} of a conic pseudo-Finslerian metric $L:\mathcal U\to \mathbb R$; and 
    \item $ \mathcal F:\mathcal U\to TM$ is an anisotropic vector field, called \textbf{external force}.
\end{enumerate}
A \textbf{motion} of the conic pseudo-Finslerian mechanical system, or \textbf{$\mathcal M$-motion}, is a $\mathcal U$-admissible smooth curve  $\gamma:I\to M$ that is a solution to \textbf{Newton's equation}
\begin{equation}\label{NewtonEq}
D^{\gamma'}_{\gamma}\gamma'=\mathcal  F(\gamma'),
\end{equation}
where $D_{\gamma}^{\gamma'}$ is the covariant derivative of $L$ along $\gamma$.  If $\gamma=(x_1,\dots,x_n)$ in a coordinate system and $\mathcal  F(v_p)=\sum_{i}F_i(v_p)\frac{\partial}{\partial x_i}(p),$ 
Newton's equations are equivalent to 
$$x_k''+\sum_{ij}\Gamma^k_{ij}(\gamma')x_i'x'_j=F_k(\gamma'),$$
for any $k\in \{1,\dots,n\}.$

\begin{remark}
     The geodesics of a conic pseudo-Finslerian metric $L$ on open subset $U\subset \mathbb R^n$ coincide with the motions of Euclidean mechanical system $(U,\frac{1}{2}\Vert.\Vert^2, \mathcal F)$ with effective inertial force
    $$\mathcal F(v_p):=- \sum_{ijk}\Gamma^k_{ij}(v_p) v_iv_j \frac{\partial}{\partial x_k}(p)$$ 
    where $\Gamma^k_{ij}$ are the Christoffel symbols of $L$. From this perspective, the geodesic equation \ref{geodesicequation} can be interpreted as Newton's second law in $\mathbb{R}^n$, with the geometrical terms (Christoffel symbols) playing the role of the components of an external force generated solely by  anisotropy of the metric.
\end{remark}

\subsection{Conservative Conic Pseudo-Finslerian Mechanical Systems}

A conic pseudo-Finslerian mechanical system $\mathcal M=(M,\mathcal L, \mathcal F)$ is  
 \textbf{conservative} if there exists a function $\phi:M \to \mathbb R$, called \textbf{potential energy}, such that  
$$ \mathcal F(v)=-\nabla^v \phi,$$
for all $v\in \mathcal U.$ The \textbf{mechanical energy} of this  mechanical system is the Lagrangian
$$\mathcal E=\mathcal L+\phi\circ \pi.$$

\begin{lemma}\label{LemaLagrangeano}
 Let $\mathcal M= (M,\mathcal L,\mathcal F)$  
 be a conservative conic pseudo-Finslerian mechanical system  with potential function $\phi$, and  let $\tilde{\mathcal L}=\mathcal L -\phi\circ \pi$.  Then, the $\mathcal M$-motion coincide with the $\tilde{\mathcal L}$-geodesics. Furthermore  $\mathcal E^{\tilde{ \mathcal L}}=\mathcal L+\phi\circ \pi$.
\end{lemma}
\begin{proof}Let $\gamma_s$ be a variation with fixed endpoints of $\gamma$ and $V$ its variational vector field. Then, by the metric compatibility of the $\nabla$ and the property $C_{\gamma'}(\gamma',.,.)=0$, it follows that
\begin{eqnarray*}
    \frac{d}{ds}|_{s=0}\int_a^b\tilde{\mathcal L}(\gamma_s')dt&=& \int_a^b\{\frac{1}{2}\frac{d}{ds}|_{s=0}L(\gamma'_s)-d\phi(V)\}dt \\
    &=&\int_a^b\{g_{\gamma'}(\frac{d}{ds}|_{s=0}\gamma_s',\gamma')-d\phi(V)\}dt\\
   &=&\int_a^b\{g_{\gamma'}(\frac{d}{dt}V,\gamma')-d\phi(V)\}dt\\
    &=& \int_a^b\{\frac{d}{dt}g_{\gamma'}(V,\gamma')-g_{\gamma'}(V,D^{\gamma'}_{\gamma}\gamma')-d\phi(V)\}dt\\
    &=& -\int_a^b\{g_{\gamma'}(V,D^{\gamma'}_{\gamma}\gamma')+d\phi(V)\}dt.
\end{eqnarray*}
    Therefore $\gamma$ is $\tilde{ \mathcal L}$-geodesic if and only if $$g_{\gamma'}(D^{\gamma'}_{\gamma}\gamma'+\nabla^{\gamma'}\phi,V)=g_{\gamma'}(V,D^{\gamma'}_{\gamma}\gamma')+d\phi(V)=0$$ for each $V$  vector field along $\gamma$.

Finally, the last statement follows directly from Lemma\ref{L-phi} and Lemma~\ref{lema01}.
\end{proof}

\begin{theorem}[Conservation of energy]  In a conservative conic pseudo-Finslerian mechanical system the mechanical energy is constant
 along any motion.
\end{theorem}
\begin{proof}
The result follows from Lemma~\ref{LemaLagrangeano} and Theorem~\ref{conservacaodaEnergia}.  
\end{proof}

Let $\mathcal M=(M,\mathcal L=\tfrac{1}{2}L,-\nabla^v \phi)$ be a  conservative conic pseudo-Finslerian mechanical system and $e\in \mathbb R$ such that 
$$M_e:=\{p\in M; e-\phi(p)>0\}\neq \emptyset.$$
The \textbf{Jacobi metric} on $M_e$ is the conic pseudo-Finslerian metric given by
$$L^J(v):=(e-\phi\circ \pi(v))L(v)$$
for all $v\in \mathcal U\cap TM_e .$

\begin{theorem}\label{Jacobitheorem}
 The motions of a conservative conic pseudo-Finslerian mechanical system $\mathcal M=(M,\mathcal L=\tfrac{1}{2}L ,-\nabla ^v\phi)$ with mechanical energy $e$ are, up to reparametrization, geodesics
 of the Jacobi metric $L^J$ on $M_e$.
\end{theorem}

\begin{proof}
Let $\gamma:I\to M$ be an admissible smooth curve. By Corollary~\ref{lema2}, 
\begin{equation} \label{eq111}
    \tilde D_{\gamma}^{\gamma'}\gamma'= D_{\gamma}^{\gamma'}\gamma'+\frac{(\Phi\circ \gamma)'}{\Phi\circ \gamma}\gamma'-\frac{g_{\gamma'}(\gamma',\gamma')}{2\Phi\circ \gamma}\nabla^{\gamma'}\Phi,
\end{equation}
where $\tilde D$ is the anisotropic covariant derivative of $ L^J$ and $\Phi(p):=e-\phi(p)$. If $\mathcal E(\gamma')$ is constant equal to $e$, then 
$g_{\gamma'}(\gamma',\gamma')=L(\gamma')=2(e-\phi(\gamma))=2\Phi\circ \gamma.$
 Thus
\begin{equation}\label{eq11}
    \tilde D_{\gamma}^{\gamma'}\gamma'= D_{\gamma}^{\gamma'}\gamma'+\frac{(\Phi\circ \gamma)'}{\Phi\circ \gamma}\gamma'-\nabla^{\gamma'}\Phi.
\end{equation}
Since 
$-d\phi=d\Phi$,
the Newton equation becomes $D_{\gamma}^{\gamma'}\gamma'=\nabla^{\gamma'}\Phi,$
because 
$$g_{\gamma'}(  D_{\gamma}^{\gamma'}\gamma',.)=-d\phi_{\gamma}(.)=d\Phi_{\gamma}(.).$$
By Equation~\ref{eq11}, $\gamma$ is an  $\mathcal M$-motion if and only if $$  \tilde D_{\gamma}^{\gamma'}\gamma'=\frac{(\Phi\circ \gamma)'}{\Phi\circ \gamma}\gamma'.$$
This concludes the proof by Lemma~\ref{Lema1}.\end{proof}
\begin{example}[Anisotropic Cosmological Drift as a Zermelo System] To illustrate the physical relevance of our construction, consider a simplified cosmological model
inspired by  Finsler–Randers anisotropic framework discussed in \cite{praveen2026finsler}. Let  $M$  be a four-dimensional spacetime slice where the background metric is given by a flat FLRW-like metric  $a_{ij}$ and let  $b = b_1(t) dt$ be a time-dependent $1$-form representing a cosmological directional drift (e.g., associated with an anisotropic dark sector or a primordial vector field). Under our framework, a conservative mechanical system with a potential  $\phi(x)$ on this conic pseudo-Finslerian space yields a modified navigation problem. By applying Theorem \ref{Jacobitheorem}, the trajectories of a test particle with fixed energy e  correspond to the geodesics of a Jacobi-Zermelo metric. Here, the external drift  $W = b^i \partial_i$  acts precisely as a "cosmological wind" that is rescaled by energy-potential factor  $\sqrt{e - \phi(x)}$. While standard cosmological approaches often rely on an osculating Riemannian metric fixed along a specific observer's four-velocity  $y(x)$ to evaluate the field equations, our formulation provides
the exact, parameter-free evolution of the particle's trajectories. The resulting Jacobi metric completely characterizes how the combined effect of the potential  $\phi(x)$  and the cosmic anisotropy $b_i(t)$ deforms the conic domain of allowed physical motions.
\end{example}

Let $L:\mathcal U\to \mathbb R$ be a conic pseudo-Finslerian metric on $M$.   Let $\gamma$ be a $\mathcal U$-admissible smooth curve and $V\in \mathfrak X(M)$ and $\mathcal L=\frac{1}{2}L$. The \textbf{momentum in the $V$-direction} is 
$$\omega_{\gamma}(V)=g_{\gamma'}(\gamma',V\circ \gamma)=\ell^{\mathcal L}(\gamma')(V\circ \gamma)$$
and the \textbf{work in the $V$-direction} is 
$$w_{\gamma}(V)=g_{\gamma'}(D_{\gamma}^{ \gamma'}\gamma',V\circ \gamma).$$
By compatibility with the metric of the Chern connection, we have 
\begin{equation}\label{eq22}
    w_{\gamma}(V)=\gamma'(\omega_{\gamma}(V))-\omega_{\gamma}(\nabla^{\gamma'}_{\gamma'}V).
\end{equation}
And if $\gamma$ is a motion of the conservative conic pseudo-Finslerian system $\mathcal M=(M,\mathcal L,-\nabla^v\phi)$, then  
\begin{equation}\label{eqdw}
    w_{\gamma}=-d\phi.
\end{equation}

\begin{theorem}[Momentum conservation] \label{MomentumConservation} Let $L$ be a conic pseudo-Finslerian metric and $\mathcal L=\frac{1}{2}L$.  
Let $\gamma$ be a motion  of a conservative conic pseudo-Finslerian mechanical system $\mathcal M =(M,\mathcal L,-\nabla^v \phi)$ and $V\in \mathfrak X(M)$ be an $L$-Killing vector field such that
    $w_{\gamma}(V)=0.$ 
Then $\omega_{\gamma}(V)$ is constant and $w_{\gamma}(\nabla_{\gamma'}^{\gamma'}V)=0$.
\end{theorem}
\begin{proof}
     Let $\varphi_s$ be the flow of
the $L$-Killing vector field $V$. As $\varphi_s$ are local $L$-isometries, each $\varphi_s$ will preserves the
Lagrangian $\mathcal L$, in particular, 
\begin{equation}
   \mathcal  L(d(\varphi_s)(\gamma'))=\mathcal L(\gamma').
\end{equation}
    As $d\phi(\gamma')=w_{\gamma}(V)=0$, $\phi$ is constant along the integral curves of $V$: 
    $$\phi(p)=\phi(\varphi_s(p)).$$ 
    Therefore, the Lagrangian $\tilde{\mathcal L}=\mathcal L-\phi\circ \pi$
    is invariant under the flow of $V$ and 
\begin{equation}
        w_{\gamma}(V)=d\phi(V)=0,\end{equation}
        by Equation~\ref{eqdw}. As $\gamma$ is a $\tilde{\mathcal L}$-geodesic (Lemma~\ref{LemaLagrangeano}), by  Lemma~\ref{Noether} and Lemma~\ref{L-phi}, $$ \ell^{\tilde{\mathcal  L}}(\gamma')(V)=\ell^{\mathcal  L}(\gamma')(V)=\omega_{\gamma}(V)$$ is constant. Therefore  we have that $w_{\gamma}(\nabla_{\gamma'}^{\gamma'}V)= \gamma'(\omega_{\gamma}(V)) -w_{\gamma}(V)=0$. 
\end{proof}

\subsection{Completeness}

In this section, we investigate the global-in-time existence of trajectories for
conic Finslerian mechanical systems. We first establish completeness criteria for
conformally deformed Finslerian metrics and Jacobi-Finslerian metrics.
We then apply these geometric results to guarantee the temporal completeness of
both conservative and dissipative mechanical motions. We introduce an analytical framework that adapts classical completeness criteria to our anisotropic setting \cite{sistemcompleteness,gordon1973analytical}. In standard Riemannian and Finslerian geometry, the global existence of geodesics (geodesic completeness) is intimately related to the behavior of proper functions with bounded gradients, which prevent curves from escaping to infinity in finite time. To extend these robust geometric techniques to the asymmetric conic domain of our mechanical system and guarantee that the $\mathcal M$-motions are defined for all $t > 0$, we define a directional analytical analogue termed a positively complete function.

\begin{proposition}\label{PropF}
If $(M,L)$ is a Finslerian manifold and $f$ is a proper function on $M$, then
$$\tilde{L}(v) = L(v) + [df(v)]^2, \quad \forall v \in TM$$
defines a complete Finslerian metric on $M$.
\end{proposition}

\begin{proof} 
The map $L_0(v,r) = L(v) + r^2$ defines a Finslerian metric on $M \times \mathbb{R}$. Let $\hat{L}$ be the Finslerian metric induced by $L_0$ on the graph $G = \{(p,f(p)) : p \in M\} \subset M \times \mathbb{R}$. Since $\varphi(p) = (p,f(p))$ defines a diffeomorphism from $M$ onto $G$, it follows that
$$ \tilde{L}(v) = \varphi^*\hat{L} = L(v) + (df(v))^2 $$
is a Finslerian metric on $M$.

Finally, let us verify that $(G, L_0)$ is complete, which implies that $\tilde{L}$ is also complete. If $\{(p_n, f(p_n))\}$ is a forward Cauchy sequence in $G$, then $\{f(p_n)\}$ is a Cauchy sequence in $\mathbb{R}$, since $L_0(v,r) \geq r^2$ for any $(v,r) \in T(M \times \mathbb{R})$. Let $z \in \mathbb{R}$ be the limit of $\{f(p_n)\}$. In particular, $\{z, f(p_1), f(p_2), \dots\}$ is a compact set. Since $f$ is a proper function, its inverse image $f^{-1}(\{z, f(p_1), f(p_2), \dots\})$ is compact. Since $\{p_1, p_2, \dots\} \subset f^{-1}(\{z, f(p_1), f(p_2), \dots\})$, it admits a convergent subsequence. Consequently, $\{(p_n, f(p_n))\}$ contains a forward convergent subsequence and is therefore forward convergent, as it is a forward Cauchy sequence.

Analogously, we prove that $(M,\tilde{L})$ is backward complete.
\end{proof}

\begin{corollary}\label{Cor1}
Let $\tilde{L}$ be a Finslerian metric. If there exists a proper function $f: M \to \mathbb{R}$ such that 
$\tilde{L} - (df)^2$
defines a Finslerian metric, then $\tilde{L}$ is complete.
\end{corollary} 

\begin{proof}
If $L = \tilde{L} - (df)^2$ is a Finslerian metric, then $\tilde{L} = L + (df)^2$ is complete by Proposition~\ref{PropF}.
\end{proof}

\begin{lemma}\label{lem1}
Let $\tilde{L}$ be a Finslerian metric. Then $L = \tilde{L} - (df)^2$ defines a Finslerian metric if and only if $g^{\tilde{L}}_v(\widetilde{\nabla}^v f,\widetilde{\nabla}^v f) < 1$ for each $v\in TM\setminus 0$, where $\widetilde{\nabla} f$ is the anisotropic $\tilde{L}$-gradient of $f$.
\end{lemma}

\begin{proof}  
Suppose first that $g^{\tilde L}_v(\widetilde{\nabla}^v f,\widetilde{\nabla}^v f) < 1$. By the Cauchy-Schwarz inequality with respect to $g^{\tilde{L}}_v$, we have
\begin{eqnarray*}
    g^L_v(w,w) &=& g^{\tilde{L}}_v(w,w) - (df)^2(w) \\
    &=& g^{\tilde{L}}_v(w,w) - g^{\tilde{L}}_{v}(\widetilde{\nabla}^v f, w)^2 \\
    &\geq& g^{\tilde{L}}_v(w,w) - g^{\tilde{L}}_{v}(w,w) g^{\tilde{L}}_{v}(\widetilde{\nabla}^v f, \widetilde{\nabla}^v f) \\
    &=& g^{\tilde{L}}_v(w,w) [1 - g^{\tilde{L}}_{v}(\widetilde{\nabla}^v f, \widetilde{\nabla}^v f)] > 0.
\end{eqnarray*}

Conversely, if $L$ is a Finslerian metric, then
\begin{eqnarray*}
    0 &< & g^L_v(\widetilde{\nabla}^v f,\widetilde{\nabla}^v f)=g^{\tilde L}_v(\widetilde{\nabla}^v f,\widetilde{\nabla}^v f) - (g^{\tilde L}_v(\widetilde{\nabla}^v f,\widetilde{\nabla}^v f))^2 \\
    &=& g^{\tilde L}_v(\widetilde{\nabla}^v f,\widetilde{\nabla}^v f)[1 - g^{\tilde L}_v(\widetilde{\nabla}^v f,\widetilde{\nabla}^v f)],
\end{eqnarray*}
which implies that $g^{\tilde L}_v(\widetilde{\nabla}^v f,\widetilde{\nabla}^v f) < 1$.

\end{proof}

\begin{corollary}\label{CorC}
Let $\tilde{L}$ be a Finslerian metric. If there exists a proper function $f$ such that $g^{\tilde{L}}_v(\widetilde{\nabla}^v f,\widetilde{\nabla}^v f) < 1$ for each $v\in TM\setminus 0$, then $\tilde{L}$ is complete.
\end{corollary} 

\begin{proof}
This follows directly from the two preceding results.
\end{proof}

\begin{proposition}\label{completudemetricaJacobi}
Let $L$ be a Finslerian metric and $\tilde{L} = (\Phi \circ \pi) L$ be a Jacobi metric of $L$. If there exists a proper function $f$ such that $g^L_v(\nabla^v f,\nabla^v f) \leq \Phi(\pi(v))$ for each $v\in TM\setminus 0$, then $\tilde{L}$ is complete.
\end{proposition} 
\begin{proof}
Since $\tilde{L}$ is positively $2$-homogeneous, equation \eqref{conformGrad} yields
\begin{eqnarray*}
    g^L_v(\nabla^v f,\nabla^v f) &=& \frac{1}{\Phi \circ \pi} g^{\tilde{L}}_v((\Phi \circ \pi)\widetilde{\nabla}^v f,(\Phi \circ \pi)\widetilde{\nabla}^v f) \\ 
    &=& (\Phi \circ \pi) g^{\tilde{L}}_v(\widetilde{\nabla}^v f,\widetilde{\nabla}^v f),
\end{eqnarray*}
which, by hypothesis, implies that $g^{\tilde{L}}_v(\widetilde{\nabla}^v f,\widetilde{\nabla}^v f) \leq 1$ for each $v\in TM\setminus 0$.  Therefore, $\tilde{L}$ is complete by Corollary~\ref{CorC}.
\end{proof}

A function $h \colon [0,\infty) \to \mathbb{R}$ is called \textbf{positively complete} if it is of class $C^1$, non-increasing, and satisfies
\[
\int_0^{\infty} \frac{ds}{\sqrt{e - h(s)}} = \infty,
\]
 for all $e$ such that $e > \sup h$.

\begin{theorem}\label{completudemovimento} 
Let $L \colon \mathcal{U} \to \mathbb{R}$ be a conic Finslerian metric, where $\mathcal{U}$ is a conic convex open set. Let $(M,\tfrac12 L,-\nabla^v \phi)$ be a conservative conic Finslerian mechanical system and let $\gamma$ be an $\mathcal{M}$-motion. Suppose that there exists an admissible proper function $H$ such that $h(t) := H(\gamma(t))$ is a positively complete function, $L(\nabla H)|_{\gamma} \leq C^2$ for some positive constant $C$, and $\phi \circ \gamma \geq h$. Then $\gamma$ is defined for all $t > 0$.
\end{theorem}

\begin{proof}
Let $\gamma \colon [0,b) \to M$ be an $\mathcal{M}$-motion and let $e = \mathcal{E}(\gamma')$ be its energy. Then 
\begin{equation}
    L(\gamma') \leq 2(e - h(t)).
\end{equation}
By the Cauchy-Schwarz inequality,
$$h' = dH(\gamma') = g^L_{\nabla H} (\nabla H, \gamma') \leq \sqrt{L(\nabla H) L(\gamma')}.$$
By hypothesis, there exists a constant $C$ such that $L(\nabla H)|_{\gamma} \leq C^2$, which yields
$$h'(t) \leq C\sqrt{L(\gamma')} \leq C\sqrt{2(e - h(t))}.$$
Consider the ODE
$$f'(t) = C\sqrt{2(e - f(t))} \quad \text{with} \quad f(0) = h(0).$$
The time required for $f$ to reach a value $s$ is given by
\[
t = \int_{f(0)}^{s} \frac{du}{C\sqrt{2(e - u)}}.
\]
Since $h$ is positively complete, the function $f$ is defined for all $t > 0$. By construction, $h' \leq f'$ and $h(0) = f(0)$, implying that $h(t) \leq f(t)$.

Furthermore, since $h$ is positively complete, there exists a constant $c$ such that $c \leq h$. Thus, for any $b \in \operatorname{Dom}(\gamma)$, we have
$$H(\gamma([0,b))) = h([0,b)) \subset \{s \in \mathbb{R} : c \leq s \leq f(b)\}.$$
Consequently, $\gamma([0,b))$ is contained in the compact set $H^{-1}([c, f(b)])$, which implies that $\gamma$ is defined for all $t > 0$.
\end{proof}

Let $L \colon \mathcal{U} \to \mathbb{R}$ be a conic pseudo-Finslerian metric on $M$. An anisotropic vector field $Y \in \mathfrak{X}(\mathcal{U})$ is said to be \textbf{dissipative} if 
$$d\mathcal{L}_v(Y) = g^L_v(v, Y(v)) < 0$$
for every $v \in \mathcal{U}$, where $\mathcal{L} = \frac{1}{2}L$ is the kinetic energy associated with $L$. The conic pseudo-Finslerian mechanical system $\mathcal{M} = (M, \mathcal{L}, \mathcal{F})$ is \textbf{dissipative} if 
$$\mathcal{F}(v) = -\nabla^v \phi \circ \pi(v) + Y(v),$$
where $\phi$ is a potential function and $Y$ is a dissipative anisotropic vector field of $(M,L)$. The \textbf{associated conservative mechanical system} is given by $\mathcal{M} = (M, \mathcal{L}, -\nabla^v \phi)$.

\begin{proposition}\label{energiasistemadissipativo}
Let $\gamma$ be a motion of a dissipative mechanical system $\mathcal{M} = (M, \mathcal{L}, \mathcal{F})$ with $\mathcal{F}(v) = -\nabla^v \phi \circ \pi(v) + Y(v)$, and let $\mathcal{E} = \mathcal{L} + \phi \circ \pi$ be the total energy of the associated conservative system. Then $\mathcal{E}(\gamma')$ is decreasing.
\end{proposition}

\begin{proof}
By definition, $D_{\gamma}^{\gamma'}\gamma' = \mathcal{F}(\gamma')$. Thus,
\begin{eqnarray*}
    \frac{d}{dt}\Big|_{t=t_0} \mathcal{E}(\gamma') &=& d\mathcal{E}_{\gamma'(t_0)}(D^{\gamma'}_{\gamma}\gamma') \\
    &=& -d\mathcal{E}_{\gamma'(t_0)}(\nabla^{\gamma'(t_0)} \phi) + d\mathcal{E}_{\gamma'(t_0)}(Y(\gamma')).
\end{eqnarray*}
By the conservation of energy and the regularity of $\mathcal{L}$, we have $d\mathcal{E}_{\gamma'(t_0)}(\nabla^{\gamma'(t_0)} \phi) = 0$. Indeed, since the Lagrangian $ \mathcal{L} - \phi\circ \pi$ associated with the conservative mechanical system $(M, \mathcal{L}, -\nabla \phi)$ is regular, there exists a motion $\alpha$ of this system satisfying $\alpha'(0) = \gamma'(t_0)$. Hence,
$$d\mathcal{E}_{\gamma'(t_0)}(\nabla^{\gamma'(t_0)} \phi) = \frac{d}{dt}\Big|_{t=0} \mathcal{E}(\alpha') = 0.$$
On the other hand, since $Y(\gamma'(t_0))$ is tangent to $T_{\gamma(t_0)}M$, it follows that
\begin{eqnarray*}
    d\mathcal{E}_{\gamma'(t_0)}(Y(\gamma')) &=& d\mathcal{L}(Y(\gamma')) + d(\phi \circ \pi)(Y(\gamma')) \\
    &=& d\mathcal{L}(Y(\gamma')) + d\phi (d\pi(Y(\gamma'))) \\
    &=& d\mathcal{L}(Y(\gamma')).
\end{eqnarray*}
Therefore, $\frac{d}{dt}\mathcal{E}(\gamma') < 0$, since $Y$ is dissipative.
\end{proof}

\begin{proposition}
Let $\mathcal{M} = (M, \frac{1}{2}L, \mathcal{F})$ be a convex conic dissipative Finslerian mechanical system and let $\gamma$ be an $\mathcal{M}$-motion. Under the hypotheses of Theorem~\ref{completudemovimento}, $\gamma$ is defined for all $t > 0$.
\end{proposition}

\begin{proof}
By Proposition~\ref{energiasistemadissipativo}, we obtain $L(\gamma'(t)) \leq 2(\mathcal{E}(\gamma'(0)) - h(t))$. Consequently, the result follows by an argument analogous to the proof of Theorem~\ref{completudemovimento}.
\end{proof}

\subsection{Conic pseudo-Finslerian mechanical system with constraints}
 A \textbf{holonomic constraint} on a conic pseudo-Finslerian mechanical system  $\mathcal M=(M,\mathcal L=\tfrac{1}{2}L ,\mathcal  F)$ is a
 submanifold $N \subset  M$ with $\dim N < \dim M$. A curve $\gamma : I \subset \mathbb R \to  M$ is said to be
 \textbf{compatible} with $N$ if $\gamma(t) \in N $ for all $t \in I$.

 A \textbf{reaction force} on a conic pseudo-Finslerian mechanical system with holonomic constraint
 $\mathcal M_N=(M,\mathcal L ,\mathcal  F, N)$ is an anisotropic vector field $\mathcal  R : \mathcal U\cap TN \to  TM$ along $N$ such that  $ \pi\circ \mathcal  R(v)=\pi(v)$   for each $v\in\mathcal U\cap  TN$ and there exists a solution $\gamma : I \subset \mathbb R \to N$ of the
 \textbf{generalized Newton equation}
 \begin{equation}
D_{\gamma}^{\gamma'}\gamma'=\mathcal  F(\gamma')+\mathcal  R(\gamma')
 \end{equation}
 with $\gamma'(0)=v$, where $D$ is the anisotropic covariant derivative of $(M,L)$. A reaction force $\mathcal  R$ is said to be \textbf{perfect}, or to satisfy the \textbf{d’Alembert principle}, if $\mathcal R(v)\in TN^{\perp}_v$ for each $v\in \mathcal U\cap TN$:
 $g_v(\mathcal R(v),w)=0$
 for all $w\in T_{\pi(v)} N$.

 \begin{theorem}\label{holonomiccontrains}
    Let
 $(M, \mathcal L =\tfrac{1}{2}L,\mathcal  F, N)$ be a conic pseudo-Finslerian mechanical system with holonomic constraint such that $N$ is a pseudo-Finslerian submanifold. Then there exists a unique  reaction force $\mathcal  R:\mathcal U\cap TN \to TM$ satisfying
 the d’Alembert principle and the solutions of the generalized Newton equation are exactly the motions of the pseudo-Finslerian mechanical system $(N,\hat L:= L|_{\mathcal U\cap TN},\mathcal  F_N)$, where $\mathcal F_N:\mathcal U\cap TN\to TN$ is an anisotropic vector field given by
 $\mathcal  F_N(v)=P_{v}(\mathcal  F(v)),$
and $P_{v}$ is the projection on $TN$ of the decomposition $TM=TN\oplus TN^{\perp}_v$.\end{theorem}

 \begin{proof}
Assume that a perfect reaction force $\mathcal  R$ exists. Then the solutions of the generalized
 Newton equation satisfy 
$D_{\gamma}^{\gamma'}\gamma'=\mathcal  F(\gamma')+\mathcal  R(\gamma').$ Therefore, $\mathcal R$ is uniquely determined by
\begin{equation}\label{EqRation}
    \mathcal R(v)=B(v)-P^{\perp}_v(\mathcal F(v))
\end{equation}
for all $v\in \mathcal U\cap  TN$, where $B$ is the operator defined by Equation~\ref{fakesegundaforma}. To establish existence, define $\mathcal  R$ via the equation above. By definition, $g_v(\mathcal R(v),w)=0$ for each $v\in \mathcal U\cap TN$ and $w\in T_{\pi(v)}N$.  Given    $v\in \mathcal U\cap TN$ let $\gamma:I\to N$ be the unique solution of ordinary differential equation 
$P_{\gamma'}(D_{\gamma}^{\gamma'}\gamma')= P_{\gamma'}(\mathcal F(\gamma'))$
 with initial condition $\gamma'(0)=v$. It follows that: 
\begin{eqnarray*}
D_{\gamma}^{\gamma'}\gamma'&=&P_{\gamma'}(D_{\gamma}^{\gamma'}\gamma')+P^{\perp}_{\gamma'}(D_{\gamma}^{\gamma'}\gamma')\\&=&P_{\gamma'}(D_{\gamma}^{\gamma'}\gamma')+P^{\perp}_{\gamma'}(\mathcal  F(\gamma')) +\mathcal R(\gamma')=\mathcal F(\gamma')+\mathcal  R(\gamma'). 
\end{eqnarray*}
  \end{proof}

A \textbf{non-holonomic constraint} on a conic pseudo-Finslerian mechanical system  $(M, L ,\mathcal  F)$ is a smooth regular distribution $\mathcal D$ on $M$. A curve $\gamma : I \subset \mathbb R \to  M$ is said to be
 \textbf{compatible} with $\mathcal D$ if $\gamma'(t) \in \mathcal D $ for all $t \in I$. We say that  $(M,\mathcal L=\tfrac{1}{2}L ,\mathcal  F, \mathcal D)$ is a \textbf{conic pseudo-Finslerian mechanical system with non-holonomic constraint}. In particular, if $\mathcal D$ is integrable, we have a foliation of $M$ by holonomic constraints. 

 A \textbf{reaction force} on
 $(M,\mathcal L=\tfrac{1}{2}L ,\mathcal  F, \mathcal D)$ is a map $\mathcal  R : \mathcal U\cap \mathcal D \to  TM$ such that  $\pi\circ  \mathcal R(v)=\pi(v)$   for each $v\in\mathcal U\cap \mathcal D$ and there is a curve $\gamma:I\to M$  with $\gamma'(0)=v$ compatible with $\mathcal D$ and solution of the
 \textbf{generalized Newton equation}
 \begin{equation}\label{EqNG}
D_{\gamma}^{\gamma'}\gamma'=\mathcal F(\gamma')+\mathcal R(\gamma'),
 \end{equation}
 where $D$ is the anisotropic covariant derivative of $(M,L)$.  A reaction force $\mathcal  R$ is said to be \textbf{perfect}, or to satisfy the \textbf{d’Alembert principle}, if $\mathcal  R(v)\in \mathcal D_v^{\perp}$ for each $v\in \mathcal U\cap \mathcal D$.

 \begin{theorem}\label{noholonomiccontrains}
 Let $(M,\mathcal L=\tfrac{1}{2}L ,\mathcal  F, \mathcal D)$ be a conic pseudo-Finslerian mechanical system with non-holonomic constraint
 such that $\mathcal D_p$ is a nondegenerate subspace for each $p\in M$. Then  there exists a unique  reaction force $\mathcal  R:\mathcal U\cap \mathcal D \to TM$ satisfying
 the d’Alembert principle.
 \end{theorem} 
 \begin{proof} As in Theorem~\ref{holonomiccontrains}, if there exists a perfect reaction force $\mathcal  R$, then
\begin{equation}\label{eq.D}
    \mathcal R(v)=B_{\mathcal D}(v)-P^{\perp}_{v}(\mathcal  F(v)).
\end{equation} 
This implies the uniqueness.  To prove the existence, define $\mathcal  R$ by Equation~\ref{eq.D}. By construction, $\mathcal  R(v)\in \mathcal D^{\perp}_v$ for each $v\in \mathcal D\cap \mathcal U$. Now note that the Equation~\ref{eq.D} 
 is equivalent to  
$P_{\gamma'}(\nabla_{\gamma'}^{\gamma'}\gamma')=
    P_{\gamma'}(\mathcal  F(\gamma'))$. By the standard theory of ordinary differential equations, there exists  a unique $\mathcal D$-compatible curve $\gamma$ that solves this last equation subject to the initial condition $\gamma'(0)=v$.
\end{proof}

\subsection{Example: Zermelo  Mechanical Systems}\label{examPSonda}

The dynamics of particles moving on a Riemannian manifold $(M,h)$ in the presence of a potential $\phi$ is modeled by the mechanical system 
$\mathcal M_0=(M,\tfrac{1}{2}h,-\nabla \phi),$
where $\nabla\phi$ denotes the $h$-gradient of $\phi$. If $W$ is a vector field with $h(W,W)<1$ that induces an anisotropy on the space, the total energy is given by 
$\mathcal E(v)=\frac{1}{2}Z(v)+\phi\circ \pi(v),$
where $Z$ is the Zermelo metric with navigation data $(h,W)$. The critical points of the action functional coincide with the motions of the conservative Finslerian mechanical system $\mathcal M=(M,\frac{1}{2}Z,-\nabla^v \phi)$. 

By Theorem~\ref{Jacobitheorem}, the motions with total energy $e$ are precisely the geodesics of the Jacobi metric
$Z^J = (\Phi\circ \pi)\, Z,$
where $\Phi(p)=e-\phi(p)$, defined on $M_e = \{p\in M :\Phi(p) > 0\}$. 
By Lemma~\ref{ZemeloConforme}, $Z^J$ is itself a Zermelo metric with navigation data $\Big((\Phi\circ \pi)h,\; \tfrac{1}{\sqrt{\Phi\circ \pi}}\,W\Big)$. Notice that the metric $(\Phi\circ \pi) h$ coincides with the Jacobi metric associated with the conservative pseudo-Finslerian mechanical system $\mathcal M_0$ at the fixed energy level $e$.

If $N$ is a holonomic constraint, then the induced mechanical system on $N$ is conservative with potential $\phi|_N$ and kinetic energy $\tilde {\mathcal Z}$, where $\tilde{\mathcal{Z}}$ is the kinetic energy of the Randers metric with navigation data $\Big( \lambda h|_{TN \times TN}, W^{\top} \Big)$, by Theorem~\ref{holonomiccontrains} and Lemma~\ref{Randersimmersion}. Consequently, the motions with energy $e$ are reparametrizations of the geodesics of the Zermelo metric with navigation data $\Big((\Phi \lambda)h|_{TN \times TN}, \frac{1}{\sqrt{\Phi}}W^{\top} \Big)$. If $\mathcal{D}$ is a non-holonomic constraint, then the motions with energy $e$ are reparametrizations of the geodesics of the sub-Finslerian metric (see \cite{alabdulsada2025sub}) $(\mathcal{D}, Z)$, where $Z_p$ is the Randers-Minkowski metric on $\mathcal{D}_p$ with navigation data 
$\Big((\Phi(p) \lambda(p))h|_{\mathcal{D}_p \times \mathcal{D}_p}, \frac{1}{\sqrt{\Phi}}W^{\top}(p) \Big)$ 
for every $p \in M$.

\end{document}